\documentclass[11pt,reqno,a4paper]{amsart}
\usepackage{comment}
\usepackage{paralist}

\usepackage{amssymb}
\usepackage{enumerate}
\usepackage[all]{xy}

\usepackage{pgf,tikz}
\usetikzlibrary{arrows}

\renewcommand{\ge}{\geqslant}
\renewcommand{\le}{\leqslant}
\renewcommand{\geq}{\geqslant}

\newcommand{\rto}{\dasharrow}

\newcommand{\PP}{\mathbb{P}}
\newcommand{\A}{\mathbb{A}}

\newcommand{\FF}{\mathbb{F}}

\newcommand{\cO}{\mathcal{O}}

\newcommand{\infnear}[1][]{>^{#1}}

\DeclareMathOperator{\kod}{kod}
\DeclareMathOperator{\overkod} {\overline {\kod}}

\newtheorem{theorem}{Theorem}[section]
\newtheorem{lemma}[theorem]{Lemma}
\newtheorem{prop}[theorem]{Proposition}
\newtheorem{corollary}[theorem]{Corollary}
\newtheorem{problem}{Problem}
\newtheorem*{problem*}{Problem}
\newtheorem{claim}{Claim}
\newtheorem{conj}[theorem]{Conjecture}

\theoremstyle{definition}

\newtheorem{remark}[theorem]{Remark}
\newtheorem{example}[theorem]{Example}

\newtheorem{property}{Property}

\title[Contractible curves on a rational surface]{Contractible curves on a rational surface}

\author{Alberto Calabri \and Ciro Ciliberto}

\email{alberto.calabri@unife.it}
\curraddr{Dipartimento di Matematica e Informatica,
Universit\`a degli Studi di Ferrara,
Via Machiavelli 30, 44121 Ferrara, Italy,
phone: +39-0532-974067, fax: +39-0532-974003}

\email{cilibert@mat.uniroma2.it}
\curraddr{Dipartimento di Matematica, Universit\`a degli Studi di Roma ``Tor Vergata'',
Via della Ricerca Scientifica, 00133 Roma, Italy, phone: +39-06-7259-4684,
fax: +39-06-7259-4699}


\begin{document}

\begin{abstract} In this paper we prove that if $S$ is a smooth, irreducible, projective, rational, complex surface and $D$ an effective, connected, reduced divisor on $S$, then the pair $(S,D)$ is \emph{contractible} (i.e., there is a birational map $\phi\colon S\dasharrow S'$ with $S'$ smooth such that $\phi_*(D)=0$) if  ${\rm kod}(S,D)=-\infty$. More generally, we even prove that this contraction is possible without blowing up an assigned cluster of points on $S$. Using the theory of peeling, we are also able to give some information in the case $D$ is not connected.  

\end{abstract}

\maketitle

\section{Introduction}

Let $(S,D)$ be a pair with $S$ a smooth, irreducible, projective, complex surface and $D$ an effective, reduced divisor on $S$. The pair $(S,D)$ is said to be \emph{contractible}  if there is a birational map $\phi\colon S\dasharrow S'$ with $S'$ smooth such that $\phi_*(D)=0$, i.e., $D$ is contracted to a finite set of points by $\phi$. The \emph{contractibility problem} consists in finding necessary and sufficient conditions for pairs  $(S,D)$ to be contractible. 

The question of characterizing contractible pairs $(S,D)$ is somehow trivial, unless $S$ is a rational surface (see Proposition \ref {prop:t} below). If $S$ is rational, the problem  has its roots in the study of \emph{Cremona geometry} of the complex projective plane $\PP^2$ (see \cite {CC}  for an historical account).  Classical results (often with incomplete proofs), in the  framework of the so called Italian school of algebraic geometry, go back to Castelnuovo--Enriques \cite{CE}, Marletta \cite {Marletta, Marletta2}, Coolidge \cite[p.~398]{Coolidge}.

The first result on the subject in modern times  is due to Kumar and Murthy in 1982, cf.\ \cite{KumarMurthy}. It can be stated as follows:

\begin{theorem}[Kumar--Murthy] \label{thm:KumarMurthy} Let $(S,D)$ be a pair with $S$ rational and $D$ smooth and irreducible. Then $(S,D)$ is contractible if and only if
the linear systems $|K_S+D|$ and $|2K_S+D|$ are both empty.\end{theorem}

Given the pair $(S,D)$, for any non--negative integer $m$, the \emph{$m$--log plurigenus} of $(S,D)$ is
\[
P_m(S,D):=h^0(S, \cO_S(m(D+K_S)).
\]
If $P_m(S,D)=0$ for all $m\ge1$, then one says that the \emph{log Kodaira dimension} of the pair $(S,D)$ is $\kod(S,D)=-\infty$. Otherwise
\[\kod(S,D)=\max\left\{ \dim \left( {\rm Im} \left( \phi_{|m(D+K_S)|} \right) \right)\right\}\]
where $\phi_{|m(D+K_S)|} $ is the rational map determined by the linear system $|m(D+K_S)|$, whenever this is not empty. One sets $\kod(S):=\kod(S,0)$, which is the  \emph{Kodaira dimension} of $S$. 

If $(S,D)$ is contractible, with $S$ rational, and $\phi\colon S\rto S'$ is a birational map which contracts $D$ to a finite set of points, there is a commutative diagram
\[
\xymatrix{
& \bar S \ar[rd]^\beta \ar[ld]_\alpha \\
S \ar@{-->}[rr]_\phi & & S'
}
\]
where $\alpha$ and $\beta$ are birational morphisms and the proper transform $\bar D$ of $D$ via $\alpha$ is contracted to a finite union of points by $\beta$. This implies that all irreducible components of $D$ have geometric genus 0 and 
that $\kod (\bar S,\bar D)=-\infty$ (see Lemma \ref {rem:bir}). 

If, in the above setting, $D$ is smooth
then one sees that $\kod(S,D)=\kod (\bar S,\bar D)=-\infty$ (see Lemma \ref {lem:mor}).  So Theorem \ref{thm:KumarMurthy} implies that if $S$ is rational and $D$ is smooth and irreducible, then $\kod(S,D)=-\infty$ if and only if  $P_2(S,D)=0$, which can be seen as a log-analogue of Castelnuovo's rationality criterion for regular surfaces. 

As for extensions of Kumar--Murthy's Theorem to reducible curves, the only known result so far was due to Iitaka \cite {Iit, Iit2}, which can be stated as follows:

\begin{theorem}[Iitaka] \label{thm:Iitaka}
Let $(S,D)$ be a pair with $S$ rational and $D$ with simple normal crossings and at most two irreducible components. Then $P_2(S,D)=0$ if and only if $\kod(S,D)=-\infty$ and, if this happens, then 
 $(S,D)$ is contractible.\end{theorem}

Concerning reducible curves,
the following theorem, though not immediately related to the contractibility problem, should also be recalled.

\begin{theorem}[Kojima--Takahashi, \cite {KT}]
Let $(S,D)$ be a pair where $S$ is a smooth rational surface and $D$  a smooth, reduced curve on $S$ with at most four irreducible components. Then, $\kod(S,D)=-\infty$ if and only if $P_6(S,D)=0$.

Furthermore, if $(\tilde S,\tilde D)$ is an almost minimal model of $(S,D)$ (see \S \ref {ssec:amin} below), and if $\tilde D$ is connected, then $\kod(S,D)=-\infty$ if and only if $P_{12}(S,D)=0$.
\end{theorem}

A classical example of Pompilj's \cite {Pompilj} (see also \cite [Example 1]{CC2}) shows that Kumar--Murthy's and Iitaka's theorems  cannot be extended, as they stand, to curves with more than two components. In  Pompilj's example one has a smooth curve $D$ on a rational surface $S$ with three irreducible components and $|K_S+D|$ and $|2K_S+D|$ both empty, but $|3K_S+D|$ non--empty. 
In this example each of the components of  $D$ is contractible, but $D$ is not. This shows the  difficulty in proving  contractibility  by induction on the number of irreducible components of the curve, as one may be tempted to do (see the historical account in \cite{CC}). The reason is that, after having contracted (if possible) some of the components of a reducible curve $D$, in order to make further contractions one may need to blow--up again points where previous components have been contracted, thus creating loops in the contraction process. 

In \cite{CC2} we posed the following:

\begin{problem}\label{prob:m}
Suppose $(S,D)$ is a pair with $S$ rational and $D$ reduced.  Then does  $\kod(S,D)=-\infty$ imply that $(S,D)$ is contractible?
\end{problem}

As  a little evidence, in  \cite {CC} we answered affirmatively to this question if $(S,D)$ is the embedded resolution of $d\geqslant 12$ distinct lines in $\mathbb P^2$. However the problem is  open in its full generality. 

The present paper is devoted to give an affirmative answer to Problem \ref {prob:m} in some cases.
After some preliminaries presented in \S  \ref {S:Iitaka}, we prove in Theorem \ref {thm:cool} that the answer  to Problem \ref{prob:m}  is affirmative if $D$ is connected.
This could be proved also as a consequence of the following deep:

\begin{theorem}[{Miyanishi-Sugie, Fujita; cf. \cite[Theorem 2.1.1] {M}}]
Let $D$ be a reduced connected divisor on a rational surface $S$ such that $\kod(S,D)=-\infty$.
Then there exists a morphism $h\colon S\setminus D \to J$, with $J$ a curve, such that any fibre of $h$ is either isomorphic to $\A^1$ or to $\PP^1$.
\end{theorem}

However, we prove more. Indeed, in \S \ref {ss:mark} we introduce the concept of a \emph{marked triple} $(S,D,\mathfrak K)$, where we add to the pair $(S,D)$ a \emph{cluster} $\mathfrak K$, i.e., a finite set of proper or infinitely near points on $S$ (see \S \ref {ssec:cluster} for a precise definition).  We define the action of birational maps on pairs $(D,\mathfrak K)$ (see again  \S \ref {ssec:cluster}) and we  introduce the concept of  \emph{contractible triples} $(S,D,\mathfrak K)$ (see \S \ref {lem:numeric}). 
Theorem  \ref {thm:cool} says that  $(S,D,\mathfrak K)$, with $S$ rational and $D$ connected, is contractible if  $\kod(S,D)=-\infty$. The extension to marked triples is motivated by the need of keeping track of 
\emph{previously contracted} components of $D$, as we mentioned above. 

The proof of Theorem  \ref {thm:cool}  uses standard techniques in surface theory. Mori's theory is however hidden in it, under the form of a lemma by Fujita's  (see \cite {Fu}). In Remark \ref {itaka} we sketch the  proof of (an extended version of)  Iitaka's Theorem \ref  {thm:Iitaka}, which is not conceptually different from the original, but is definitely shorter. In \S \ref {sec:appl} we give a couple of applications. 

The assumption of connectedness of $D$ plays a central role in the proof of Theorem \ref {thm:cool}  and we have been unable to do without it. However we have been trying a different approach to the problem, which, though not 	exhaustive, gives some information even in the non--connected case. Indeed we prove a different contraction criterion, i.e., Theorem \ref {thm:main}, in \S \ref {sec:peelthm}. This is based on  Miyanishi--Tsunoda's \emph{theory of peeling} (see \cite {MT, M}), which we briefly recall in 
\S \ref {sec:peel} for the reader's convenience. Theorem \ref {thm:main} basically says that if $(S,D,\mathfrak K )$ is a marked triple, with $(S,D)$ almost minimal, $S$ rational and $\kod(S,D)=-\infty$, then $(S,D)$ is contractible unless, perhaps, either there is a birational morphism $\phi\colon S\to \bar S$, contracting $D$ to (singular) points, with $\bar S$ a normal logarithmic del Pezzo surface of rank 1 (this is called a \emph{logarithmic del Pezzo surface of rank 1 with shrinkable boundary}, see the definition in \S \ref {ssec:barkc}), or $D$ contains one, and only one, very specific connected component, called a \emph{non--admissible fork} (see \S\ref {ssec:def}). 

The classification of logarithmic del Pezzo surfaces of rank 1 is still an open problem in its generality (see Conjecture \ref {conj:MT} below). Keel and McKernan gave in \cite {KM} a classification theorem (Theorem 23.2), which applies to all but a bounded family of rank one logarithmic del Pezzo surfaces.
A case by case analysis (most likely to be quite hard) based on  Keel--McKernan's results could possibly shed some more (though not decisive) light on the resolution of Problem \ref {prob:m}.

\subsection*{Notation} We set, as usual, $\FF_n:=\PP(\mathcal O_{\PP^ 1}\oplus \mathcal O_{\PP^ 1}(-n))$, with the structure morphism $f_n:\FF_n\to \PP^ 1$. We will denote by $E$ an irreducible section of $f_n$ with $E^ 2=-n$ (which is unique if $n\neq 0$), and by $F$ a fibre of $f_n$. One has $K_{\FF_n}\equiv -2E-(n+2)F$. 

We use the symbol $\equiv$ for linear equivalence. If $f\colon S\to S'$ is a birational morphism between smooth surfaces and $P$ is a point of $S'$ where $f^ {-1}$ is not defined, then one can consider the \emph{$(-1)$--cycle} $D$ on $S$ (usually a non--reduced divisor)  contracted to $P$ by $f$: $D$ is 1-connected, $K_S\cdot D=D^ 2=-1$. 

The concept of $1$--connected effective divisors on a smooth surface is well known and we freely use it. For the rest, we use standard notation and concepts in algebraic geometry. 
\section{Preliminaries} \label{S:Iitaka}

Let $S$ be a smooth, irreducible, projective, complex surface and $D$ an effective divisor on $S$. The \emph{support} ${\rm Supp}(D)
$ of $D$ is the reduced divisor sum of the irreducible components of $D$. 
The divisor $D$ will usually be for us non--zero and reduced, i.e., $D={\rm Supp}(D)$ in which case $D$ is called a \emph{curve}. We will often consider the case in which the curve $D$ has \emph{simple normal crossing singularities} (shortly, $D$ is snc), i.e., each component of $D$ is smooth and $D$ has at most nodes. In this case we will say that the pair $(S,D)$ is \emph{log smooth}.  

\subsection{Infinitely near points}\label{ssec:infnear}
Let $S$ and $S'$ be smooth, irreducible, projective surfaces.
Any birational morphism $\sigma\colon S'\to S$
is the composition of a certain number $n$
of blow--ups $\sigma_i\colon S_{i}\to S_{i-1}$
at a point $P_i\in S_{i-1}$ with exceptional divisor $E_i$ on $S_i$, for $i=1,\ldots,n$
\begin{equation}\label{eq:seq1}
\sigma \colon S'=S_{n}
  \xrightarrow{\,\sigma_n\,}  S_{n-1}
  \xrightarrow{\,\sigma_{n-1}\,}
  \cdots
  \xrightarrow{\,\sigma_{2}\,}  S_{1}
  \xrightarrow{\,\sigma_{1}\,}  S_0=S.
\end{equation}

Let $P\in S$ be a point.
        One says that $Q$ is an \emph{infinitely near point to $P$ of (vicinity) order $n$} on $S$,
and we write $Q\infnear[n] P$ (or simply $Q\infnear P$ if $n$ is understood), if there exists
a birational morphism $\sigma\colon S'\to S$ as in \eqref{eq:seq1},
such that $P_1=P$, $\sigma_i(P_{i+1})=P_{i}$, $i=1,\ldots,n-1$,
and $Q\in E_n$. Points of vicinity order 0 are
the points of $S$, which are called \emph{proper points}. We denote by $\mathfrak P(S)$
the set of infinitely near points on $S$ and, abusing terminology, we refer to $\mathfrak P(S)$ as the set of \emph{points on} $S$. 


Given a curve $C$ on $S$, one says that it \emph{passes through the point $Q\infnear P$}, with $P\in S$, $Q\in S'$ and $\sigma\colon S'\to S$ as above, if 
the proper transform $C'$ of $C$ on $S'$ passes through $Q$. One also says that $Q$ is infinitely near to $P$ \emph{along} $C$. The notion of infinitely near points which are base points for a linear system of curves on $S$ is then clear. 

\subsection{Clusters}\label{ssec:cluster}

In this paper a \emph {cluster} $\mathfrak K$ on $S$ is a finite subset of $\mathfrak P(S)$
(note that in \cite{Alberich} a different definition of cluster is used).
The \emph{support} ${\rm Supp}(\mathfrak K)$ of $\mathfrak K$ is the set of proper points $P$ such that $Q\infnear P$ for $Q\in \mathfrak K$.  The points of ${\rm Supp}(\mathfrak K)$ are not required to be in $\mathfrak K$.  A cluster is \emph{simple} if ${\rm Supp}(\mathfrak K)$ consists of one point. Every cluster is a finite union of simple clusters. The concept of a curve passing through a cluster is clear. 
Clusters have a partial ordering $\mathfrak K\leqslant \mathfrak K'$ which is given by containment. The \emph{(vicinity) order} of a cluster is the maximum order of a point in the cluster.

Given a cluster $\mathfrak K$ on $S$, there is a surface $S_\mathfrak K$ and a birational morphism 
$\phi_{\mathfrak K}\colon S_\mathfrak K\to S$, where each point of the cluster has been blown--up and 
$S_\mathfrak K$ is minimal under this condition. We will denote by $E_\mathfrak K$ the  sum of all  proper transforms on $S_\mathfrak K$ of the exceptional divisors of blow--ups of points of $\mathfrak K$. 
The cluster $\mathfrak K$ on $S$ is determined by the triple $(S_\mathfrak K, E_\mathfrak K, \phi_{\mathfrak K})$. 

\begin{remark}\label{rem:a}  Let $Z$ be a zero--dimensional scheme on $S$. There is a smooth surface $S'$, a birational morphism $\phi\colon S'\to S$ and a divisor $E$ on $S'$ such that $\phi_*(\mathcal O_{S'}(-E))=\mathcal I_{Z,S}$.  Then ${\rm Supp}(E)$ determines a cluster $\mathfrak K_Z$ on $S$, called the \emph{supporting cluster} of $Z$.  Note that $E$ is reduced if and only if $Z$ is \emph{curvilinear}, i.e., $Z$ is a subscheme of a smooth curve $C$ on $S$. In this case $Z$ is uniquely determined by $\mathfrak K_Z$. \end{remark}

\begin{remark}\label{rem:aa}  Given a triple $(S',E',\phi)$, with $\phi\colon S'\to S$ a birational morphism and $E'$ an effective, reduced divisor on $S'$ which is contracted to a union of points by $\phi$, there is a unique  cluster $\mathfrak K$ on $S$ and a unique birational morphism 
$\varphi_\mathfrak K: S' \to S_\mathfrak K$ such that $\phi=\phi_\mathfrak K\circ \varphi_\mathfrak K$, 
$\varphi_\mathfrak K(E')=E_\mathfrak K$, and no component of $E'$ is contracted to a point by $\varphi_\mathfrak K$. 
In particular $S_\mathfrak K$ is uniquely determined up to isomorphisms. \end{remark}

Let $f\colon S\dasharrow S'$ be a birational map between smooth, irreducible, projective surfaces. This induces a birational map $f_\mathfrak K\colon S_\mathfrak K\dasharrow S'$. 
We have a diagram
\begin{equation}\label{eq:res}
\xymatrix{
& \tilde S \ar[rd]^\beta \ar[ld]_\alpha \\
S_\mathfrak K \ar[rd]_{\phi_{\mathfrak K}} \ar@{-->}[rr]^{f_\mathfrak K} & & S'\\
& S \ar@{-->}[ru]_f}
\end{equation}
where $\alpha$ and $\beta$ are sequences of blow--ups. Let  $\tilde E_\mathfrak K$ be the proper transform of $E_\mathfrak K$ on $\tilde S$. We set
\[
{\rm div}_f(\mathfrak K)=\beta_*(\tilde E_\mathfrak K)
\]
which is called the \emph{divisorial part} of the \emph{image} of $\mathfrak K$ via $f$. Let $\tilde E_{\mathfrak K,0}$ be the maximal subdivisor of $\tilde E_\mathfrak K$ contracted to points by $\beta$. By Remark \ref {rem:aa}, this  determines a cluster $\mathfrak K'$ on $S'$. We set
\[
{\rm cl}_f(\mathfrak K)=\mathfrak K'
\]
called the \emph{image cluster} of $\mathfrak K$. 

We will say that $f$ \emph{does not blow--up the cluster} $\mathfrak K$ if ${\rm div}_f(\mathfrak K)=0$.

\begin{remark}\label{rem:bb} Morphisms do not blow--up any cluster.  Moreover, if $f$ does not blow--up $\mathfrak K$ and $\mathfrak K'\leqslant \mathfrak K$, then $f$ does not blow--up $\mathfrak K'$.  \end{remark}

\subsection {Markings}\label {ss:mark} Let $S$ be a smooth, irreducible, projective surface. A \emph{marked pair} on $S$ is a pair $(D,\mathfrak K)$ with $D$ an effective, reduced divisor on $S$ and $\mathfrak K$ is a cluster on $S$, called the \emph{marking} of the pair, whereas $D$ is the \emph{divisorial part} of the pair.
The triple $(S,D,\mathfrak K)$ will be called a \emph{marked triple}.  If the cluster $\mathfrak K$ is empty, we write  $(S,D)$ instead of $(S,D,\mathfrak K)$.  

Markings have a partial ordering
\[
(D,\mathfrak K)\leqslant (D',\mathfrak K')\quad \stackrel{\rm def}\Longleftrightarrow \quad D\leqslant D' \quad \text{and}\quad \mathfrak K\leqslant \mathfrak K'.
\]

Let $f: S\dasharrow S'$ be a birational map and let $(D,\mathfrak K)$ be a marking. We want to define $f_*(D,\mathfrak K)$ which will be a marked pair $(D',\mathfrak K')$ on $S'$, and we will then say that $f$ maps $(S,D, \mathfrak K)$ to $(S', f_*(D,\mathfrak K))$. 

We have a diagram
\[
\xymatrix{
& \tilde S \ar[rd]^\beta \ar[ld]_\alpha \\
S \ar@{-->}[rr]_{f} & & S'}
\]
with $\alpha, \beta$ sequences of blow--ups. We let $\tilde D$ be the proper transform of $D$ on $\tilde S$. We set
\[
D'=\beta_*(\tilde D)+{\rm div}_f(\mathfrak K). 
\]
Let $\tilde D_0$ be the maximal subdivisor of $\tilde D$ contracted to points by $\beta$. By Remark \ref {rem:aa}, this determines a cluster $\mathfrak D$ of $S'$, and we define
\[
\mathfrak K'={\rm cl}_f(\mathfrak K)\cup \mathfrak D. 
\]

\subsection{Cremona transformations} Let $S$ be a smooth, irreducible, projective surface. Consider a rational dominant map $f\colon S\dasharrow X$, with $X$ an irreducible, projective variety. If  $X$ is non--degenerate in $\PP^ r$,  there is a fixed components free linear system $\mathcal L$ of dimension $r$ on $S$ such that $\mathcal L$ is the pull--back via $\phi$ of the hyperplane linear system of $\PP^ r$ and  $f$ coincides with the map $\phi_\mathcal L$ determined by $\mathcal L$. 

\begin{example}\label{ex:nb} Let $f\colon S\dasharrow S'$ be a birational map between smooth, irreducible, projective surfaces and assume that $f=\phi_\mathcal L$.  Let $\mathfrak K$ be a cluster on $S$. We want to give conditions under which $f$ does not blow--up $\mathfrak K$. We assume $\mathfrak K$ simple, with support $P\in S$ (the non--simple case can be treated similarly). By Remark \ref{rem:bb} we may assume that $P\in \mathfrak K$.

If $f$ is a morphism, it does not blow--up $\mathfrak K$.  Hence,  
if $P$ is not a base point for $\mathcal L$, then $f$ does not blow--up $\mathfrak K$. 

Suppose that $P$ is a base point for $\mathcal L$, that the general curve in $\mathcal L$ is smooth at $P$ and that there is a smooth curve $C$ through $P$ such that the intersection multiplicity of $C$ with the general curve in $\mathcal L$ is $m>0$. In other words $P$, and its subsequent infinitely near points along $C$ up to order $m$, are base points for $\mathcal L$. We claim that, if $m$ is larger that the order of $\mathfrak K$, then  $f$ does not blow--up $\mathfrak K$. 

Indeed, \eqref {eq:res} specifies as follows 
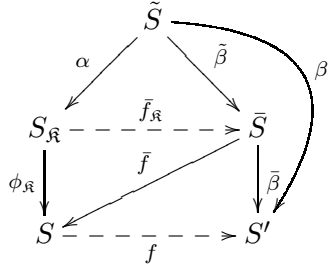
\[
\xymatrix{
&\tilde S \ar[ld]_\alpha \ar[dr]^{\tilde\beta} \ar@/^3.5pc/[ddr]^ \beta\\
S_\mathfrak K\ar[d]_{\phi_\mathfrak K}
\ar@{-->}[rr]^{\bar f_\mathfrak K}
&&\bar S  \ar[lld]_{\bar f} \ar[d]^{\bar\beta} \\
S\ar@{-->}[rr]_f & & S'
}
\]
where:\\
\begin{inparaenum}
\item [(i)] the lower right triangle is the resolution of the indeterminacies of $f$, hence it 
is the composition of the blow--ups of $P$ and of its subsequent infinitely near points along $C$ up to order $m$, plus perhaps other blow--ups at points $P'$ with $P'\not \infnear P$. Hence the curve on $\bar S$ contracted by $\bar f$ to $P$ is a chain $E$ of rational curves $E_1+\cdots+E_m$, where
$E_i\cdot E_j=\delta_{i,j-1}$, $1\leqslant i<j\leqslant m$, and $E_i^ 2=-2$ if   $i=1,\ldots, m-1$ whereas $E_m^2=-1$;\\
\item [(ii)] $f_\mathfrak K=f\circ \phi_\mathfrak K$ and $\bar f_\mathfrak K$ is determined by $f$;\\
\item [(iii)] since the order of $\mathfrak K$ is smaller than $m$, the curves in $\tilde E_{\mathfrak K}$
are either contracted by $\tilde \beta$ to points  or are mapped to one of the curves $E_1,\ldots,E_{m-1}$;\\
\item [(iv)]  $\bar \beta$ contracts  $E_1+\cdots+E_{m-1}$ to a point. 
\end{inparaenum}

Our claim follows from (iii) and (iv).   

\end{example}

\begin{example} \label {ex:eltr} We recall the \emph{elementary transformations} of surfaces $\FF_n$. Pick a point $P\in \FF_n$ and consider the diagram
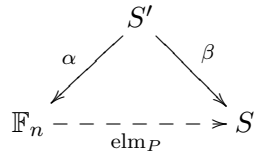
\[
\xymatrix{
& S' \ar[rd]^\beta \ar[ld]_\alpha \\
\FF_n \ar@{-->}[rr]_{{\rm elm}_P} & & S}
\]
where $\alpha$ is the blow--up of $\FF_n$ at $P$ and $\beta$ is the contraction of the proper transform 
of the fibre of $|F|$ passing through $P$ on $S'$. Then 
\[ S=
\Bigg \{  \begin{array}{ccc}
\FF_{n+1} &\text {if}& P\in E \\
\FF_{n-1} &\text {if} &P\not\in E.\\
\end{array}
\] 
The map ${\rm elm}_P$ is called the \emph{elementary transformation based at $P$}.

More generally, one can make elementary transformations when we have a surface $S$ and a base point free pencil $|F|$ of rational curves.\end{example}

If $S$ is rational, we have a birational map $\phi: S\dasharrow \PP^2$, which is determined by a linear system $\mathcal L$ of dimension 2 whose general element  is an irreducible curve of geometric genus 0 and two general curves of $\mathcal L$ intersect transversely at one point off the base points of $\mathcal L$. Any such a linear system is called a \emph{homaloidal net}. 

In particular,  \emph{Cremona transformations} of $\PP^ 2$, i.e., birational maps $\phi: \PP^ 2\dasharrow \PP^ 2$, are of the 
form $\phi_\mathcal L$, with $\mathcal L$ homaloidal nets of plane curves. The Cremona transformation $\phi_\mathcal L$ is said to be \emph {based} at the base locus scheme of the homaloidal net $\mathcal L$. 

\begin{example} \label {ex:DJ} Consider on $\FF^ n$ a complete linear $|E+dF|$ of sections of $\FF_n\to \PP^ 1$. If $d\geqslant n$,  then $|E+dF|$ is base point free, of dimension $2d-n+1$ and self--intersection $2d-n$, and the general curve in $|E+dF|$ is smooth, irreducible and rational.

Let us fix $C\in |E+dF|$ smooth, and an effective divisor $D$ of degree $2d-n-1$ on $C$. Let $\mathcal L$ be the linear system of curves in $|E+dF|$ which cut out on $C$ a divisor containing $D$. Then  $\mathcal L$ is a homaloidal net determining a birational map $\phi_\mathcal L: \FF_n\dasharrow \PP^ 2$ and $\mathcal L$ has the curvilinear base locus scheme $D$ considered as a subscheme of $C$, which is determined by its supporting cluster $\mathfrak K_D$ (see Example \ref {rem:a}).
In the case $n=1$, we have a diagram
\[
\xymatrix{
& \FF_1 \ar[rd]^{\phi_\mathcal L} \ar[ld]_ \pi \\
\PP^ 2 \ar@{-->}[rr]_{\gamma} & & \PP^ 2}
\]
where $\pi$ is the blow--down of the curve $E$ to a point $P\in \PP^ 2$. The map $\gamma$ is determined by a homaloidal net of plane curves of degree $d$ with multiplicity $d-1$ at $P$ and a further curvilinear base locus scheme, or cluster, of degree $2d-2$.
A Cremona map of  type $\gamma$ is called a \emph{De Jonqui\`eres transformation}. The case $d=2$ is the case of \emph{quadratic transformations}. 
\end{example}

\subsection{Curves on a surface} Consider a pair $(S,D)$ with $D$ a curve on $S$. 

We will write $\ell(D)$ (or simply $\ell$) to denote the number of irreducible components of $D$. We will denote by $\mathfrak G(D)$ the  \emph{vector weighted graph} of $D$, i.e., the graph:\\\begin{inparaenum}
\item with $\ell$ vertices $\mathfrak d_1,\ldots, \mathfrak d_\ell$ corresponding to the irreducible components $D_1,\ldots, D_\ell$ of $D$, each with  vector weight   $(D_i^ 2,p_a(D_i))$, for $1\leqslant i\leqslant \ell$;\\
\item  for each pair of indices $(i,j)$ such that $1\leqslant i<j\leqslant \ell$ and for each intersection point $p$ of $D_i$ and $D_j$, there is an edge $\mathfrak e_{i,j,p}$ joining $\mathfrak d_i$ and $\mathfrak d_j$, with weight given by the intersection multiplicity of $D_i$ and $D_j$ at $p$, so that the sum of the weights of edges joining $d_i$ and $d_j$ is $m_{ij}=D_i\cdot D_j$. 
\end {inparaenum}

In the above setting, an edge of weight $w$  has to be considered as the superposition of $w$ simple edges, hence it contributes $w$ to the valency of the vertices it joins and contributes to the  homology of $\mathfrak G(D)$. 

If the irreducible components of $D$ are all smooth and rational, one may omit the second component of the vector weight of the vertices. We will sometimes denote by the same symbol $\mathfrak G(D)$ the unweighted graph. 

A connected curve $D$ is said to be a \emph{tree} if $\mathfrak G(D)$ is a tree, i.e., all edges have weight 1 and $h^ 1(\mathfrak G(D),\mathbb C)=0$ (then $D$ is snc).  A component of $D$  corresponding to a vertex of valency 1 of $\mathfrak G(D)$ is called \emph{terminal}. 

 The curve $D$ is \emph{connected} if and only if $\mathfrak G(D)$ is connected, in which case one has $p_a(D)\geqslant 0$.

Note that $\mathfrak G(D)$ does not identify $D$, even if all components of $D$ are smooth. For example, if $D$ consists of three lines in the plane, then $\mathfrak G(D)$ is the complete graph on three vertices, regardless to the fact that the three lines pass or do not pass through the same point. 

\begin{lemma}\label{lem:conn} Let $D$ be an effective, non--zero, reduced, connected divisor with $h$ irreducible components on a smooth surface $S$. Then:\\
\begin {inparaenum}
\item [(i)] if $p_a(D)=0$ then $\mathfrak G(D)$ is a tree, all components of $D$ are smooth and rational and $D$ has $h-1$ nodes and no other singularity  (in particular $D$ is snc);\\
\item [(ii)] if $p_a(D)=1$ then:
\begin{inparaenum}

\item [(a)]  either $D$ has a component of arithmetic genus 1 (which can be either smooth, or rational nodal, or rational cuspidal), all other components are smooth rational and $D$ has (further) $h-1$ nodes and no other singularity;

\item [(b)] or all components of $D$ are smooth and rational and $D$ has  $h$ nodes and no other singularity;

\item [(c)] or all components of $D$ are smooth and rational, $D$ has a tacnode and  $h-2$ nodes and no other singularity; 

\item [(d)] or all components of $D$ are smooth and rational, $D$ has an an ordinary triple point and 
$h-3$ nodes and no other singularity.

\end{inparaenum}
In case (a) the graph $\mathfrak G(D)$ is a tree, in all other cases it has a unique cycle. 
\end{inparaenum}
\end{lemma}

\begin{proof} The proof is standard, so we only sketch it in case (ii), giving for granted case (i), which can be proved similarly. 

The assertion is clear if $h=1$: in this case (a) occurs. So we assume $h>1$ and proceed by induction on $h$. Since $D$ is connected, there are singular points on $D$. 
Let $P$ be one of them, and let $m$ be its multiplicity. Let $\pi: S'\to S$ be the blow--up of $P$ with exceptional divisor $E$, and take the proper transform $D'$ of $D$. One has 
\begin{equation}\label{eq:p}
p_a(D')=p_a(D)-\frac {m(m-1)}2=1-\frac {m(m-1)}2.
\end{equation}

Assume first  $P$ is not a node. If $D'$ is connected, then $p_a(D')\geqslant 0$, hence $m=2$, $p_a(D')=0$, and we can apply part (i) to $D'$. Since $m=2$, then $E\cdot D'=2$. Since $P$ is not a node, then $E$ intersects $D'$ at only one point $Q$ with intersection multiplicity 2. If  $Q$ is a smooth point of $D'$, then $D$ has a cusp, and we are in case (a). If $Q$ is a node of $D'$, then $D$ has a tacnode, and we are in case (c).  If $D'$ is not connected,  let $k$ be the number of its connected components. One has $k\leqslant m$, hence $p_a(D')\geqslant -k+1\geqslant -m+1$. By \eqref {eq:p}, one has $m\leqslant 3$. The case $m=2$ is not possible, because we assumed $D'$ not connected and $P$ not  a node.  Therefore $m=k=3$,  $P$ is an ordinary triple point and, by applying to the three connected components of $D'$ part (i), we see we are in case (d).

Suppose now $P$ is a node, hence $m=2$ and $p_a(D')=0$. If $D'$ is connected then, applying part (i), we see we are in cases (a) or (b). If $D'$ is not connected, then, since $D$ is connected, $D'$ consists of two connected components $D'_1, D'_2$. Then $D=D_1+D_2$, with $D_i=\pi_*(D'_i)$, with $1\leqslant i\leqslant 2$, and $D_1$ and $D_2$ intersect transversally at  $P$. Since $1=p_a(D)=p_a(D_1)
+p_a(D_2)$, and $p_a(D_1),p_a(D_2)$ are both non--negative, we may assume that $p_a(D_1)=1,p_a(D_2)=0$. Then we conclude by applying part (i) for $D_2$ and induction for $D_1$. \end{proof}

\begin{corollary}\label{lem:log} Let $(S,D)$ be a pair with $S$ a smooth, regular  surface, and $D$ an effective, reduced, non--zero divisor, such that $|K_S+D|=\emptyset$ (which is the case if ${\rm kod}(S,D)=-\infty$). Then $(S,D)$ is log smooth. 
\end{corollary}
\begin{proof} The hypotheses yield that each connected component of $D$ has $p_a(D)=0$. The assertion follows from Lemma \ref {lem:conn}.  \end{proof}

The following lemmata will be useful:

\begin{lemma}\label{lem:numeric} Let $(S,D)$ be a pair with $S$ a smooth, rational surface, $D=D_1+\cdots+D_h$ an effective, reduced divisor, where $D_1,\ldots,D_h$ are the connected components of $D$. Assume that $p_a(D_i)=0$ for $1\leqslant i\leqslant h$. Then\\
\begin{inparaenum}[(i)]
\item $D\cdot (D+K)=-2h$;\\
\item $h^ 0(S,\cO_S(2K+D))+h^ 0(S,\cO_S(-K-D))\geqslant K\cdot (K+D)-h+1=(D+K)^2+h+1$.
\end{inparaenum}
\end{lemma}
\begin{proof} Part (i) is adjunction formula and part (ii) is Riemann--Roch theorem plus part (i).
\end{proof}

Let $(S,D)$ and $(\bar S,\bar D)$ be pairs. If there is a birational morphism $\phi: \bar S\to S$ such that  $\phi_*(\bar D)=D$, we write $(S,D)\leqslant (\bar S,\bar D)$. 

\begin{lemma}\label{lem:mor} Let $(S,D)$ and $(\bar S,\bar D)$ be pairs with $D$ and $\bar D$ effective and reduced. If $(S,D)\leqslant (\bar S,\bar D)$, then  $\kod(S,D)\geqslant \kod(\bar S,\bar D)$, i.e., $\kod(S,D)$ is a decreasing function. In particular, if $\kod(S,D)=-\infty$ then also $\kod(\bar S,\bar D)=-\infty$.
\end{lemma}
\begin{proof} It suffices to prove the assertion for the blow--up $\phi: \bar S\to S$ at a single point $P$, with exceptional divisor $E$. Then $\bar D=\phi^*(D)-aE$, with $a\geqslant 0$, hence
$\bar D+K_{\bar S}\equiv \phi^ *(D+K_S)-(a-1)E$. So the assertion is clear if $a\geqslant 1$. 

Assume $a=0$.
Then $m(\bar D+K_{\bar S})\equiv \phi^*(m(D+K_S))+mE$ and we claim that in this case $P_m(S,D)=P_m(\bar S,\bar D)$ for all positive integers $m$, hence $\kod(S,D)= \kod(\bar S,\bar D)$. Indeed, if $m(\bar D+K_{\bar S})$ is not effective, then also $\phi^*(m(D+K_S))$ is not effective, hence $P_m(S,D)=P_m(\bar S,\bar D)=0$. If $m(\bar D+K_{\bar S})$ is effective, then
$m(\bar D+K_{\bar S})\cdot E=-m$, hence $mE$ is a fixed component of $|m(\bar D+K_{\bar S})|$, and again $P_m(S,D)=P_m(\bar S,\bar D)$.
\end{proof}

\begin{remark} \label{rem:mor} The proof of Lemma \ref {lem:mor}  shows that if $0\leqslant a\leqslant 1$, then 
$P_m(S,D)=P_m(\bar S,\bar D)$ for all positive integers $m$, hence $\kod(S,D)= \kod(\bar S,\bar D)$. 

In particular, in  the setting of the proof of   Lemma \ref {lem:mor}, if $D$ has a double point
and $\bar D=\phi^ *(D)-E$, then $\bar D$ contains $E$, has the same number of connected components of $D$, and $P_m(S,D)=P_m(\bar S,\bar D)$ for all positive integers $m$. In particular $\kod(S,D)=-\infty$ if and only if  $\kod(\bar S,\bar D)=-\infty$. 
\end{remark}

Given $(S,D)$ as in Lemma \ref {lem:mor}, we set
\begin{equation}\label{eq:kodbar}
\overline {\kod}(S,D)=\min\{ \kod(\bar S,\bar D): (\bar S,\bar D)\geqslant (S,D)\, \}.
\end{equation}
By Remark \ref {rem:mor}, the minimum in \eqref {eq:kodbar} is reached once we resolve the singularities of $D$.

\subsection{Contractible triples} Let $(S,D,\mathfrak K)$ be a marked triple. One says that $(S,D,\mathfrak K)$ is \emph{contractible} if there is a birational map $f: S\dasharrow S'$  such that $f_*(D,\mathfrak K)$ has zero divisorial part. Similarly, one defines the concept of a \emph{contractible pair} $(S,D)$.  

\begin{example}\label{ex:line} A triple $(\mathbb P^ 2, L, \mathfrak K)$, where $L$ is a line and $\mathfrak K$ is any cluster, is contractible, via a quadratic transformation based at two general points of $D$.

Similarly, a triple $(\mathbb P^ 2, D, \mathfrak K)$, where $D$ is a reduced conic and $\mathfrak K$ is any cluster, is contractible. If $D$ is irreducible, we can reduce to the line case with a quadratic transformation based at three general points of $D$. If $D$ is reducible, we also reduce to the line case, by applying a quadratic transformation based at two general points of one component of $D$ and at another general point of the other component.
\end{example}

\begin{remark}\label{rem:bir}
Let $(S,D)$ and $(S',D')$ be pairs. Suppose there is a birational map $\phi: S\dasharrow S'$ such that $\phi$ [resp.\ $\phi^ {-1}$]  does not contract any irreducible component of $D$ [resp.\ of $D'$] and $\phi_*(D)=D'$ (hence 
$\phi^ {-1}_*(D')=D$). Then $(S,D)$ is contractible if and only if  $(S',D')$ is. Given this, there is no restriction, in the contractibility problem, to assume $D$ to be snc or even  smooth. 
\end{remark}

\begin{lemma}\label{lem:ob} If $(S,D)$ is contractible, then $\overkod (S,D)=\kod(S)$. Moreover, all irreducible components of $D$ have geometric genus 0. 
\end{lemma}

\begin{proof} If $(S,D)$ is contractible, there is a commutative diagram
\[
\xymatrix{
& \bar S \ar[rd]^\beta \ar[ld]_\alpha \\
S \ar@{-->}[rr]_\phi & & S'
}
\]
where $\alpha$ and $\beta$ are birational morphisms. If 
$\bar D$ is the proper transform of $D$ via $\alpha$, then $\bar D$ is contracted to a union of points by $\beta$, thus, by Lemma \ref {rem:mor}, the assertion follows.  \end{proof}

The following is obvious. 

\begin{lemma}\label{lem:sup} Let  $(S,D,\mathfrak K)$ and $(S',D',\mathfrak K')$ be triples, let $f: S\dasharrow S'$ be a birational map, such that $f_*(D,\mathfrak K)\leqslant (D',\mathfrak K')$ and assume that  $(S',D',\mathfrak K')$ is contractible. Then $(S,D,\mathfrak K)$ is contractible.  \end{lemma}

\begin{prop}\label{lem:DJ} Let $(\FF_n,D,\mathfrak K)$ be a marked triple such that
$D=\varepsilon E+D'$, with $0\leqslant \epsilon\leqslant 1$, and  $D'\in |E+dF|$. Then $(\FF_n,D,\mathfrak K)$ is contractible.
\end{prop}

\begin{proof} By taking into account Examples \ref {ex:DJ} and \ref {ex:line} and by Lemma \ref {lem:sup} the assertion is clear if $\varepsilon=0$. So we focus on the case $\varepsilon =1$. 

We may write
\[
D'= F_1+\ldots+F_k+C\quad \text{with}\quad k\geqslant 0,
\]
$F_1,\ldots, F_k\in |F|$ distinct and $C\in |E+(d-k)F|$ smooth and irreducible. We can make a series of elementary transformations based at general point of $F_1,\ldots, F_k$ and contract them.  After having done this, the proper transform of $E$ could have non--negative self intersection. However,  we can make another series of elementary transformations either based at general points of the surface or at general points of the proper transform of $E$, so to reduce to the case $k=0$ and $n=1$, where $D=E+C$ with $E^ 2=-1$ and $C\in |E+dF|$ smooth, irreducible. 

Let $n_1Q_1+\ldots+n_hQ_h$ be the degree $d-1$ divisor cut out by $E$ on $C$, with  $Q_1,\ldots, Q_h$ distinct. 
Take  non--negative integers $k_1,\ldots, k_h$, such that  $m:=k_1+\ldots+k_h+1>d$, consider the linear system $|E+mF|$ of dimension $2m$ and its sublinear system $\mathcal L$ consisting of the curves:\\
\begin{inparaenum}
\item [$\bullet$]  cutting out on $C$ a divisor containing $(n_j+k_j)Q_j$, for all $j=1,\ldots, h$;\\
\item [$\bullet$] passing through $m-d$ further general points of $E$.\\
\end{inparaenum}
The total number of base points imposed to $\mathcal L$ is $2m-2$ so that $\dim(\mathcal L)=2$ and $\mathcal L$ is a homaloidal net, hence $\phi_\mathcal L$ birationally maps $\FF_1$ to $\PP^ 2$. It maps $C$ to a line, maps $E$ to a point $P$ (which is not on $C$). By Example \ref {ex:nb}, we see that, by taking $k_1,\ldots, k_h$ sufficiently large, the map 
$\phi_\mathcal L$ does not blow--up the cluster $\mathfrak K$. So we are reduced to the case $(\PP^ 2, L, \mathfrak C)$, where $L$ is a line and $\mathfrak C$ is a suitable cluster, which is contractible by Example \ref {ex:line}. \end{proof}

Recalling Example \ref {ex:DJ}, as an immediate consequence we have:

\begin{corollary}\label{cor:plane} Let $(\PP^ 2,D,\mathfrak K)$ where $D$ is a reduced curve of degree $d$ with a point of multiplicity at least $d-1$.  Then $(\PP^ 2,D,\mathfrak K)$ can be mapped  via a De Jonqui\`eres transformation to a triple $(\PP^ 2,L,\mathfrak C)$, where $L$ is a line, and so  $(\PP^ 2,D,\mathfrak K)$ is contractible. \end{corollary}

We have also:

\begin {corollary}\label{lem:minimal} Let $(S,D,\mathfrak K)$ be a marked triple, with $S$ a minimal rational surface and $\kod(S,D)=-\infty$. Then $(S,D,\mathfrak K)$ is contractible.
\end{corollary}

\begin {proof} We assume $D$ non--zero, otherwise there is nothing to prove.

If $S=\PP^ 2$, then $\kod(S,D)=-\infty$ implies $D\equiv kL$, with $1\leqslant k\leqslant 2$ and the assertion follows by Example \ref {ex:line}. 

If $S=\FF_n$,  we have $D\equiv aE+bF$, with $a,b\geqslant 0$ and  $a+b>0$. 

If $n=0$, then $\kod(S,D)=-\infty$ implies that either $0\leqslant a\leqslant 1$ or $0\leqslant b\leqslant 1$. We may assume that $0\leqslant a\leqslant 1$, then the assertion follows  from  Proposition \ref {lem:DJ}.  

Assume next $n\geqslant 2$.  If $a\leqslant 1$, the assertion follows again from Proposition \ref {lem:DJ}.   Suppose $a\geqslant 2$. If $b\geqslant n+2$, then $K_{S}+D\equiv (a-2)E+(b-n-2)F$ is effective, a contradiction.  So  $b\leqslant n+1$. Suppose that $b<n(a-1)$, then  $D\cdot E=-an+b<-n$ which implies that $2E$ is contained in $D$, a contradiction, since $D$ is reduced. Therefore we have $n+1\geqslant b\geqslant n(a-1)$, hence we must have $a=2$. But then $D\cdot E=-2n+b<0$, so $E$ splits off $D$ and we may apply Proposition \ref {lem:DJ} to conclude.   \end{proof}

\subsection {Small pairs} \label {ssec:small} Let $(S,D)$ be a pair as above. We will often write $K$ to denote a canonical divisor $K_S$ of $S$. 

We say that the pair $(S,D)$ is \emph{small} if there is no $(-1)$--curve $E$ on $S$ such that $\epsilon:=E\cdot D\leqslant 1$.  Since $D$ is reduced, one has $\epsilon\geqslant -1$, with equality if and only if  $E$ is a connected component of $D$, in which case we say that $E$ is an \emph {isolated component} of $D$. 

Let $E$ be a $(-1)$--curve offending smallness, let $\pi: S\to S'$ be the contraction of $E$ to a point $P'$ and let $\pi_*(D)=D'$. 

If $E$  is not contained in $D$, then $D'$ is isomorphic to $D$ and $D'$ has multiplicity $\epsilon$ in $P'$. If $E$ is contained in $D$, one has $E\cdot (D-E)=\epsilon+1$, hence one has the following different possibilities:\\
\begin{inparaenum} 
\item [(i)] $\epsilon =-1$, i.e., $E$ is an isolated component of $D$, then $D'$ is isomorphic to $D-E$;\\
\item [(ii)] $\epsilon=0$, then $E$ intersects $D-E$ at a smooth point (i.e., $E$ is a {terminal} component of $D$), hence $D'$ is isomorphic to $D-E$ and $P'$ is  a smooth point of $D'$;\\
\item [(iii)] $\epsilon=1$ and $E$ intersects $D-E$ at two distinct smooth points, hence $D'$ acquires a node at $P'$;\\
\item [(iv)] $\epsilon=1$ and $E$ intersects $D-E$ at a point $P$ with intersection multiplicity 2, and $D-E$ has a \emph{cusp of order} $k\geqslant 1$ at $P$ (i.e., $D-E$ has, in a suitable neighborhood of $P$, local equation of the form $y^ 2=x^ {2k+1}$), in which case $D'$ has at $P'$ a cusp of order $k+1$;\\
\item [(v)] $\epsilon=1$ and  $E$ intersects $D-E$ at a double point $P$ with intersection multiplicity 2, and $D-E$ has a \emph{tacnode of order} $k\geqslant 1$ at $P$
(i.e., $D-E$ has in a suitable neighborhood of $P$ local equation of the form $y^ 2=x^ {2k}$), in which case $D'$ has at $P'$ a tacnode of order $k+1$.
\end{inparaenum}

\begin{lemma}\label{lem:generality} In the above setting, one has:\\
\begin{inparaenum}
\item [(i)] $D'$ has the same number of connected components of $D$, unless $E$ is an isolated component of $D$, in which case $D'$ has one component less than $D$;\\  
\item [(ii)]  $\pi^ *(K_{S'}+D')=K_S+D-(1-\epsilon) E$, in particular
$\kappa(S,D)=-\infty$ implies $\kappa(S',D')=-\infty$;\\
\item [(iii)] $p_a(D)=p_a(D')$, unless $E$ is an isolated component of $D$, in which case 
 $p_a(D)=p_a(D')-1$. 
\end{inparaenum}
\end{lemma}
\begin{proof} Parts (i) and (ii) are obvious. As for part (iii), the assertion is clear if $E$ is not contained in $D$, because then $D$ and $D'$ are isomorphic. If $E$ is contained in $D$ the assertion follows from part (ii) and adjunction formula. \end{proof}

By iterating contractions of curves offending smallness, one arrives at a small pair 
$(S_\sigma,D_\sigma)$, where  $\pi: S\to S_\sigma$ is a birational morphism and $D_\sigma=\pi_*(D)$. The pair $(S_\sigma,D_\sigma)$ is called a \emph{small model} of $(S,D)$.

\begin{remark}\label{rem:unique} A small model of $(S,D)$ is in general not unique, since it may depend on the $(-1)$--curves which one contracts first. For instance if $D=E_1+E_2$, with $E_1,E_2$ two $(-1)$--curves such that $E_1\cdot E_2=1$, then we may either contract $E_1$ or $E_2$, and the two resulting surfaces are obtained one from the other by an elementary transformation in a pencil of rational curves.
\end{remark}

We will need to keep track of the components of $D$ which is necessary to contract in order to come to a small model $(S_\sigma,D_\sigma)$. This datum is  encoded in the cluster $\mathfrak K_\pi$ determined by $\pi: S\to S_\sigma$ and by the curves contracted by $\pi$ (recall Remark \ref {rem:aa}), or rather, in the marked triple $(S_\sigma,D_\sigma,\mathfrak K_\pi)$.    

More generally, one can start with a marked triple $(S,D,\mathfrak K)$. If $(S,D)$ is small, also the triple
$(S,D,\mathfrak K)$ will be said to be \emph{small}. In any case, let  $(S_\sigma,D_\sigma)$ be a small model of $(S,D)$, with $\pi: S\to S_\sigma$ (observe that, since $\pi$ is a morphism, it does not blow--up $\mathfrak K$). We define $\mathfrak K_\sigma$ to be the union of  ${\rm cl}_\pi(\mathfrak K)$ and of $\mathfrak K_\pi$. The small triple $(S_\sigma,D_\sigma,\mathfrak K_\sigma)$ will be said to be  a \emph{small model} of  $(S,D,\mathfrak K)$. 

We finish by observing that the contractibility problem for a pair $(S,D)$ is somehow trivial if $S$ is not rational. Indeed, we have:

\begin{prop}\label{prop:t} Let $(S,D)$ be a pair with $S$ not rational and let $(S_\sigma,D_\sigma)$ be a small model. Then $(S,D)$ is contractible if and only if:\\
\begin{inparaenum}
\item [(i)]  either $D_\sigma=0$ and  $S_\sigma$ is minimal, which is the case if $\kod(S)\geqslant 0$,\\
\item [(ii)] or, only if $\kod(S)=-\infty$, each irreducible component of $D_\sigma$ is contained in a fibre of the Albanese morphism of $S_\sigma$.
\end{inparaenum}
\end{prop}

\begin{proof} If either (i) or (ii) hold, then clearly $(S_\sigma,D_\sigma)$ is contractible (in case (ii) use elementary transformations), hence also  $(S,D)$ is. Assume next that $(S,D)$ is contractible. 

First suppose $\kod(S)\geqslant 0$. Let $\phi: S\dasharrow S'$ be a birational map such that $\phi_*(D)=0$. Consider the commutative diagram
\[
\xymatrix{
& S\ar[d]_{p}
\ar@{-->}[r]^{\phi} &S'  \ar[d]^{p'}\\
&\,\,\,\,\,\,\,\,\,\,\,\,\,S_{\rm min}\,\,\,\,\,\, \ar@{->}[r]_f & \,\,\,\,\,\, S_{\rm min}
}
\]
where $p$ and $p'$ are the birational morphisms to the unique minimal model $S_{\rm min}$ and $f$ is an automorphism of $S_{\rm min}$.  Since $(p'\circ \phi)_*(D)=0$, then also
 $(f\circ p)_*(D)=0$, hence $p_*(D)=0$. This means that each connected component of $D$ is contained in some $(-1)$--cycle on $S$ and the assertion follows.
 
Assume now $\kod(S)=-\infty$. Since $S$ is not rational, then the Albanese morphism factors through a morphism $a: S\to C$, with $C$ a smooth curve of positive genus. Since all irreducible components of $D$ have geometric genus 0 (see Lemma \ref {lem:ob}), the assertion follows. 
\end{proof}

\section{The contraction theorem}\label {sec:thm}

In this section we will prove the following:

\begin{theorem}\label{thm:cool} Let $(S,D,\mathfrak K)$ be a triple with $S$ rational, $D$ connected and such that $\kod(S,D)=-\infty$. Then $(S,D,\mathfrak K)$ is contractible. 
\end{theorem}

\begin{proof}  By Lemma \ref {lem:generality} we may and will assume that $(S,D,\mathfrak K)$ is small and $D$ is non--zero, otherwise there is nothing to prove. 
The hypotheses $D$ connected and $\kod(S,D)=-\infty$ imply that $p_a(D)=0$, so that Lemma \ref {lem:conn}, (i) applies. 

By Corollary \ref {lem:minimal}, we may assume that $S$ is not minimal, so there is a $(-1)$--curve $E$ on $S$, which, by the smallness of $(S,D)$, is such that $D\cdot E\geqslant 2$. 

\begin{claim}\label{cl:00} The divisor $D+E$ is $1$-connected. 
\end{claim}

\begin{proof}[Proof of the Claim] The assertion is clear if $D$ does not contain $E$. 
Write $D+E=A+B$ with $A,B$ both effective, non--zero. There are two possibilities: \begin{inparaenum} \item [(i)] $A\geqslant 2E$ and $B\not\geqslant E$, so that $A=A'+2E$ and $A'\not\geqslant E$; \item [(ii)] $A\geqslant E$ and $B\geqslant E$, so that $A=A'+E$, $B=B'+E$ and  and $A'\not\geqslant E$,  $B'\not\geqslant E$. \end{inparaenum}

In case (i) we have $D=A'+B+E$. Then $A\cdot B=(A'+2E)\cdot B=(A'+E)\cdot B+E\cdot B\geqslant 1$, because $(A'+E)\cdot B\geqslant 1$ and $E\cdot B\geqslant 0$.

In case (ii) we have  $D=A'+B'+E$ hence $E\cdot (A'+B')\geqslant 1$ and we may assume $E\cdot B'\geqslant 1$. 
Then $A\cdot B=(A'+E)\cdot(B'+E)= A'\cdot (B'+E)+E\cdot (B'+E)=A'\cdot (B'+E)+E\cdot B'-1\geqslant A'\cdot (B'+E)\geqslant 1$.
\end{proof}

By Fujita's Lemma (see \cite {Fu}), there is a non--negative integer $m$ such that
\begin{equation}\label {eq:fuj}
\vert E+m(K+D)\vert \neq \emptyset, \,\,\, \text{and}\,\,\, \vert E+(m+1)(K+D)\vert = \emptyset.
\end{equation}

\begin{claim}\label{cl:0} One has $m>0$. 
\end{claim}
\begin{proof}[Proof of the Claim] One has $p_a(E+D)=p_a(E)+p_a(D)+D\cdot E-1\geqslant 1$ and
$D+E$ is 1--connected by Claim \ref {cl:00}.  Hence $h^ 0(D+E,\omega_{D+E})=p_a(D+E)\geqslant 1$. The adjunction exact sequence
\[
0\to \mathcal O_S(K)  \to \mathcal O_S(K+D+E)\to \omega_{D+E}\to 0
\]
yields  $h^ 0(S,O_S(K+D+E))=h^ 0(D+E,\omega_{D+E})>0$, proving the assertion. \end{proof}

We take an effective divisor $C\in \vert E+m(K+D)\vert$. \medskip

\noindent {\bf Case  $C=0$.} Then $E\equiv -m(K+D)$, hence $-1=m^ 2(D+K)^ 2$, so that $m=1$, i.e., $-K\equiv D+E$, hence  $D\cdot E=-E\cdot (K+E)=2$. 

Let $\pi: S\to S'$ be a series of blow--downs of $(-1)$--curves, the first one being $E$, with $S'$ minimal. Let  $D'=\pi_*(D)$, which is a connected reduced anticanonical divisor, singular at $\pi(E)$. Since $\pi$ is a birational morphism, it does not blow--up $\mathfrak K$, hence $\pi_*(D,\mathfrak K)=(D',\mathfrak K')$. 

Since $S'$ is minimal, one has either  $S'=\PP^ 2$ or  $S'=\mathbb F_n$, with $n\neq 1$. 
If $S'=\PP^ 2$, then $D'$ is a singular cubic and Corollary \ref {cor:plane} implies that $(\PP^ 2, D', \mathfrak K')$ is contractible. This  proves the theorem in this case. Assume now $S'=\mathbb F_n$, with $n\neq 1$.  If $n\neq 0$, we make a series of elementary transformations based at general points of a component of $D'$ which is not in $|F|$, and reduce to the case $n=0$. 
If $n=0$, the linear system $\mathcal L$ of curves in $|E+F|$ with a base point general on a component of $D'$, is a homaloidal system, the birational map  
$\phi_\mathcal L:   \mathbb F_0\dasharrow \mathbb P^ 2$ maps $D'$  to a singular curve of degree 3, and we conclude as in the case $S'=\PP^ 2$.

\medskip

\noindent {\bf Case  $C\neq 0$.}
 By Lemma \ref {lem:numeric} and since $|2(K+D)|=\emptyset$, one has
\begin{equation*}\label{eq:K}
C\cdot K=-1+mK\cdot (D+K)\leqslant -1 + mh^ 0(S,\cO_S(-K-D)).
\end{equation*}
If $C\cdot K\geqslant 0$, then $|-K-D|\neq \emptyset$, i.e., $-K-D\equiv A$, with $A$ effective and non--zero, because $D+K$ is not effective. On the other hand $C\equiv E+m(K+D)\equiv E-mA$, thus $m=0$ against Claim \ref {cl:0}. 

So $C\cdot K< 0$, hence there is an irreducible component $M$ of $C$ such that $M\cdot K<0$. 

\begin{claim}\label{cl:n} All components of $C$ are smooth and rational. Moreover
 $M^ 2\geqslant 0$ and $M\cdot D\leqslant 1$. 
\end{claim}

\begin{proof}[Proof of the Claim] The first assertion follows from $|C+D+K|=\emptyset$.

In order to prove the rest, we first remark that it cannot be the case that $M\cdot D\geqslant 2$ and $M+D$ is 1--connected. Indeed, in this case one has 
\begin{equation}\label{eq:k}
h^ 0(M+D,\omega_{M+D})=p_a(M+D)=p_a(M)+p_a(D)+M\cdot D-1\geqslant 1.
\end{equation} Then the adjunction exact sequence
\begin{equation}\label{eq:h}
0\to \mathcal O_S(K)  \to \mathcal O_S(K+M+D)\to \omega_{M+D}\to 0
\end{equation}
yields  
\begin{equation}\label{eq:hh}
h^ 0(S,\mathcal O_S(K+D+C))\geqslant h^ 0(S,\mathcal O_S(K+D+M))=h^ 0(M+D,\omega_{M+D})>0,
\end{equation} a contradiction. 

If $M^ 2\geqslant 0$, then $\dim (|M|)\geqslant 1$, hence there is some irreducible curve $M'\equiv M$ which is not contained in $D$.  Then $M\cdot D\leqslant 1$ follows. Otherwise one has $M\cdot D\geqslant 2$ and $M+D$ is 1--connected, leading, as we saw, to a contradiction.

Next we argue by contradiction and assume $M^ 2<0$, hence $M^ 2=-1$ because $K\cdot M<0$. 

If $M$ is not contained in $D$, the same argument as above implies that  $M\cdot D\leqslant 1$, against the smallness assumption. 
Hence we may assume that $D$ contains $M$ simply, because $D$ is reduced. 
By smallness, we have $M\cdot D\geqslant 2$. Then we claim that $M+D$ is 1--connected, leading again to a contradiction. 

To prove that $M+D$ is 1--connected, write $M+D=A+B$, with $A,B$ effective and not zero. Note that 
$M+D$ contains $M$ with multiplicity 2. 
If $2M$ is contained in $A$, then $M$ is not contained in $B$, and we write $A=2M+A'$, and $D=M+A'+B$. Since $D$ is connected, we have $(M+A')\cdot B\geqslant 1$ and 
\[
A\cdot B=(2M+A')\cdot B=(M+A')\cdot B+M\cdot B\geqslant 1.
\]
Otherwise, $M$ is contained simply in both $A$ and $B$, and we write $A=M+A'$, $B=M+B'$, $D=M+A'+B'$, and $A'\cdot B'\geqslant 0$, because $A',B'$ have no common component. One has $2\leqslant M\cdot D=M^ 2+M\cdot (A'+B')$, hence $M\cdot (A'+B')\geqslant 3$. Then
\[
A\cdot B=(M+A')\cdot (M+B')=M^ 2+M\cdot (A'+B')+A'\cdot B'\geq 2.
\] \end{proof}

We can now conclude the proof of the theorem. If $M^ 2=0$, then $|M|$ is a base point free pencil of rational curves which determines a morphism
$\varphi_{|M|}: S\to \PP^ 1$. 

If $M\cdot D=0$, then $D$, which is connected, is contained in a fibre. By making 
$\varphi_{|M|}: S\to \PP^ 1$ \emph{relatively minimal} (i.e., all fibres isomorphic to $\PP^ 1$), 
 we have a birational morphism $f: S\to \FF_n$ for some $n\in \mathbb N$, which does not blow--up $\mathfrak K$. Either $f$ contracts $D$ to a point, and we are done, or it maps $D$ to a fibre $F$ of $\FF_n\to \PP^ 1$, in which case we can still contract it by making an elementary transformation based at a general point  of $F$.

 If $M\cdot D=1$, then $D=\sigma+\mathfrak f_1+\ldots+\mathfrak f_h$, where $\sigma$ is a section
of $\varphi_{|M|}: S\to \PP^ 1$ and $\mathfrak f_1,\ldots,\mathfrak f_h$ are disjoint connected components in different fibres of $\varphi_{|M|}: S\to \PP^ 1$. By making 
$\varphi_{|M|}: S\to \PP^ 1$ \emph{relatively minimal}, 
 we have a birational morphism $f: S\to \FF_n$ for some $n\in \mathbb N$, which does not blow--up $\mathfrak K$, and $f_*(D,\mathfrak K)=(f_*(D),\mathfrak K')$, with $f_*(D)$ consisting of 
 a section $D'$ of $f_n:\FF_n\to \PP^ 1$ and of distinct fibres $F_1,\ldots,F_k$, with $k\leqslant h$. 
 By applying Proposition \ref {lem:DJ}, we see that $(\FF_n,f_*(D),\mathfrak K')$ is contractible.
 
 If $M^ 2>0$, then  $|M|$ is base point free and the morphism $\varphi_{|M|}$ birationally maps $S$ to a minimal rational surface without blowing--up $\mathfrak K$. Since $M\cdot D\leqslant 1$, then either $\varphi_{|M|}$  contracts $D$ or it maps $D$ to a line (plus perhaps points) and the contractibility of $(S,D, \mathfrak K)$ follows again.\end{proof}
 
 \begin{remark}\label{itaka}  The proof of Theorem \ref {thm:cool} can be easily adapted to prove also (a stronger version of) Iitaka's Theorem \ref  {thm:Iitaka}, to the effect that if $(S,D,\mathfrak K)$ is a marked triple with $S$ rational and $D$ with snc, at most two irreducible components and $\kod(S,D)=-\infty$, then  $(S,D,\mathfrak K)$ is contractible. We briefly sketch the  argument.
 
As in the proof of Theorem \ref {thm:cool}, we may assume that $(S,D,\mathfrak K)$ is small.
We may suppose $D$ consists of two irreducible component $D_1,D_2$, with $D_1\cdot D_2=0$, otherwise Theorem \ref {thm:cool} applies.  The curves $D_1$ and $D_2$ are smooth and rational. Again we may assume there is a $(-1)$--curve $E$ such that $D\cdot E\geqslant 2$.

If $D_1\cdot E=D_2\cdot E=1$, we consider the contraction $f: S\to S'$  of $E$. Then $D':=f_*(D)$ is connected and $\kod(S',D')=-\infty$ by Remark \ref {rem:mor}. So we apply Theorem \ref {thm:cool} to $(S',D')$ and finish. 

Next we can assume that $E\cdot D_1\geqslant 2$. Then $p_a(D_1+E)>0$, hence $h^ 0(S, \mathcal O_S(K+D+E))\geqslant h^ 0(S, \mathcal O_S(K+D_1+E))>0$.  Consider $m$ as in \eqref {eq:fuj}. Since $h^ 0(S, \mathcal O_S(K+D+E))>0$, Claim \ref  {cl:0} holds. 

Take again $C\in \vert E+m(K+D)\vert$, all irreducible components of which are smooth and rational. The discussion of the case $C=0$ goes through as in the proof of Theorem \ref {thm:cool}. In the case $C\neq 0$, we still find an irreducible component $M$ of $C$ such that $M\cdot K<0$.  Since $0=h^ 0(S, \mathcal O_S(K+D+C))\geqslant h^ 0(S, \mathcal O_S(K+D+M))$, then $|K+D+M|$ is empty. This implies that one of the following occurs:\\
\begin{inparaenum}
\item [(i)] $M$ coincides with one  of the curves $D_1,D_2$;\\
\item [(ii)] $M$ is distinct from $D_1$ and $D_2$ and $M\cdot D\leqslant 1$;\\
\item [(iii)] $M$ is distinct from $D_1$ and $D_2$ and $M\cdot D_1=M\cdot D_2=1$.
\end{inparaenum}
 
In case (i) assume $M=D_1$. Then $D_1^ 2\geqslant 0$, otherwise $D_1$ is a $(-1)$--curve, against smallness. By blowing--up points of $D_1$ we may assume $D_1^ 2=0$. Consider $\varphi_{|D_1|}: S\to \mathbb P^ 1$ so that $D_1$ is a fibre and $D_2$ is contained in another fibre. By making $\varphi_{|D_1|}$ relatively minimal and operating with suitable elementary transformations, we dispose of this case.

In case (ii), we have $M^ 2\geqslant 0$ by smallness. As above we may assume that $M^ 2=0$. Consider $\varphi_{|M|}: S\to \mathbb P^ 1$, which either maps $D_1$ and $D_2$ to curves in fibres or to a curve in a fibre plus a unisecant curve to the fibres. By making $\varphi_{|M|}$ relatively minimal, making elementary transformations and finally applying Proposition \ref {lem:DJ} we finish in this case. 

In case (iii), we may assume again $M^ 2\geqslant 0$, otherwise we contract $M$ and apply Theorem \ref {thm:cool} as above. Again we may assume that $M^ 2=0$. 
Then $\varphi_{|M|}: S\to \mathbb P^ 1$  maps $D_1$ and $D_2$ to two unisecant curves to the fibres. By making $\varphi_{|M|}$ relatively minimal and operating with elementary transformations, we can obtain a birational map  $f: S\dasharrow \FF_1$ mapping $D_1$ to $E$ and $D_2$ to another unisecant. Then, applying again Proposition  \ref {lem:DJ} we finish in this case too.  
\end{remark}
 
 \section{Applications}\label {sec:appl}
 
In \cite[Proposition 4.9]{CC} we proved the following proposition by induction on the degree $d$ and by using quadratic transformations. Theorem \ref {thm:cool} allows us to give a faster and more conceptual proof.

\begin{prop}
Let $C$ be the union of $d\ge4$ distinct lines in $\PP^2$ with a point $P_0$ of multiplicity $d-2$ and $2d-3$ nodes.
Then there exists a plane Cremona transformation which contracts $C$ to a set of points.
\end{prop}

\begin{proof}
Denote by $L_1,\ldots,L_{d-2}$ the lines through $P_0$ and let $L_{d-1}, L_d$ be the other two lines.
Set $P_{i,j}=L_i\cap L_j$ for $i\ne j$.
Blow up 
\[
\text{
$P_0$, $P_{1,d-1}$, $P_{2,d-1}$, \dots, $P_{d-2,d-1}$.
}
\]
Denoting by $\tilde L_i$ the strict transform of $L_i$, $i=1,\ldots,d$, on the blown-up surface $S$, it follows that
\[
\text{
$\tilde L_1^2=\tilde L_2^2=\cdots=\tilde L_{d-2}^2=-1$, \,\, $\tilde L_{d-1}^2=3-d$, \,\, $\tilde L_d^2=1$
}
\]
and $\tilde L_1,\ldots,\tilde L_d$ meet as in the following picture
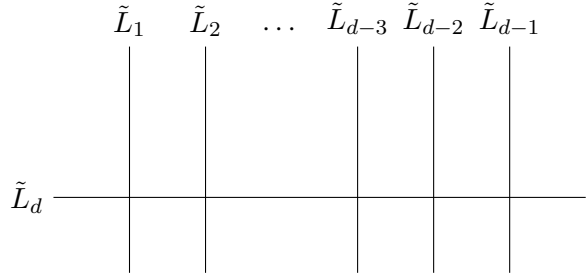
\[
\begin{tikzpicture}[scale=1.0]
    \draw (7,0) -- (0,0) node[left] {$\tilde L_d$};
    \draw (1,-1) -- (1,2) node[above] {$\tilde L_1$};
    \draw (2,-1) -- (2,2) node[above] {$\tilde L_2$};
    \draw (3,2) node[above] {$\cdots$};
    \draw (4,-1) -- (4,2) node[above] {$\tilde L_{d-3}$};
    \draw (5,-1) -- (5,2) node[above] {$\tilde L_{d-2}$};
    \draw (6,-1) -- (6,2) node[above] {$\tilde L_{d-1}$};
\end{tikzpicture}
\]
Therefore, $D=\tilde L_1\cup\cdots\cup\tilde L_d$ is connected and  $D$ belongs to the strict transform on $S$ of  the linear system of plane curves of degree $d$ with multiplicity at least $d-2$ at $P_0$ and multiplicity at least 2 at the $d-2$ points $P_{1,d-1}$, $P_{2,d-1}$, \dots, $P_{d-2,d-1}$, therefore $|m(D+K_S)|$ is the strict transform of the linear system of plane curves of degree $md-3m$ with multiplicity $md-3m$ at $P_0$ and multiplicy $m$ at each one of the $d-2$ points $P_{1,d-1}$, $P_{2,d-1}$, \dots, $P_{d-2,d-1}$, that is empty for each $m>0$.
This says that the hypotheses of Theorem \ref{thm:cool} are satisfied. Hence, $(S,D)$ is contractible.
\end{proof}

According to Proposition 4.11 in \cite{CC}, a union of $d\ge9$ distinct lines in $\PP^2$ with a point of multiplicity $d-3$ and $3(d-2)$ nodes is not contractible to a set of points by a plane Cremona transformation. The next proposition shows that, somehow unexpectedly, such a configuration of lines is instead contractible to a set of points by a plane Cremona transformation when $d\le8$.

\begin{prop}
If $C$ is a union of $d\le8$ distinct lines with a point $P_0$ of multiplicity $d-3$ and $3(d-2)$ nodes, then there exists a plane Cremona transformation which contracts $C$ to a set of points.
\end{prop}

\begin{proof}
It suffices to show the assertion for $d=8$.
Denote by $L_1,\ldots,L_5$ the lines through the point $P_0$ of multiplicity $m_0=5$
and by $L_6,L_7,L_8$ the other three lines.
Set $P_{i,j}=L_i\cap L_j$ for $i\ne j$.
Blow up $P_0$ and
\begin{equation}\label{points}
\text{
$P_{1,7}$, $P_{1,8}$, $P_{2,7}$, $P_{3,6}$, $P_{4,6}$, $P_{4,8}$, $P_{5,6}$, $P_{5,8}$, $P_{6,7}$, $P_{6,8}$ and $P_{7,8}$.
}
\end{equation}
Denoting by $\tilde L_i$ the strict transform of $L_i$, $i=1,\ldots,8$, it follows that
\[
\text{
$\tilde L_1^2=\tilde L_4^2=\tilde L_5^2=-2$, $\tilde L_2^2=\tilde L_3^2=-1$, $\tilde L_6^2=\tilde L_8^2=-4$,  $\tilde L_7^2=-3$
}
\]
and $\tilde L_1,\ldots,\tilde L_8$ meet as in the following picture
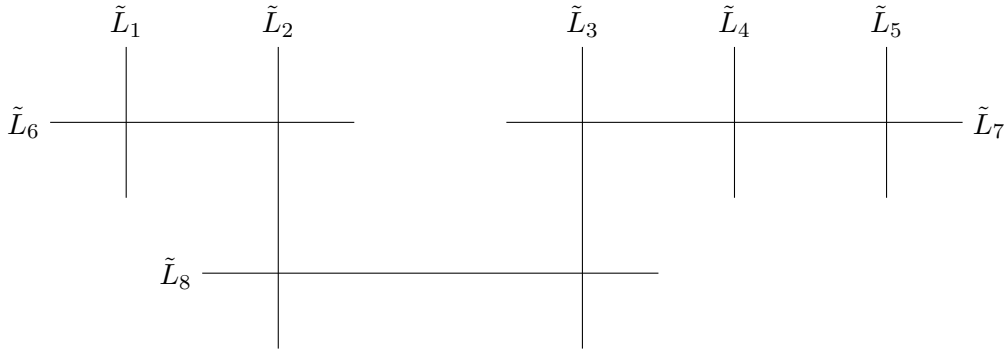
\[
\begin{tikzpicture}[scale=1.0]
    \draw (8,0) -- (2,0) node[left] {$\tilde L_8$};
    \draw (4,2) -- (0,2) node[left] {$\tilde L_6$};
    \draw (1,1) -- (1,3) node[above] {$\tilde L_1$};
    \draw (3,-1) -- (3,3) node[above] {$\tilde L_2$};
    \draw (7,-1) -- (7,3) node[above] {$\tilde L_3$};
    \draw (9,1) -- (9,3) node[above] {$\tilde L_4$};
    \draw (11,1) -- (11,3) node[above] {$\tilde L_5$};
    \draw (6,2) -- (12,2) node[right] {$\tilde L_7$};
\end{tikzpicture}
\]
Therefore, $D=\tilde L_1\cup\cdots\cup\tilde L_8$ is connected and $D$ sits in the strict transform on $S$ of the linear system of plane curves of degree 8 with multiplicity at least 5 at $P_0$ and multiplicity at least 2 at the eleven points \eqref{points}.
Hence $|m(D+K_S)|$, with $m$ a positive integer, is the strict transform of the linear system $\mathcal L_m$ of plane curves of degree $5m$ with multiplicity at least $4m$ at $P_0$ and multiplicity at least $m$ at the eleven points \eqref{points}. We claim that the system $\mathcal L_m$ is empty for all positive integers $m$, proving that $\kod(S,D)=-\infty$ hence that $(S,D)$ is contractible. To prove the claim (it suffices to assume $m$  divisible enough), we notice that $(D+K_S)\cdot \tilde L_i=-1$ for $i=1,4,5$. This implies that the lines $L_i$ with $i=1,4,5$ split off with multiplicity $\frac m2$ from $\mathcal L_m$. At this point the line $L_6$ splits off with multiplicity $\frac m 6$. Then the line $L_1$ splits off again with multiplicity $\frac m{12}$, $L_2$ with multiplicity $\frac m6$, and the line $L_{0,7,8}$ joining $P_0$ and $P_{7,8}$ with multiplicity $\frac m6$. This splitting process goes on and after a few steps the residual system becomes empty because its degree becomes negative. We defer the reader who is interested in the full computation to the link \texttt{http://docente.unife.it/alberto.calabri/adjoint1.txt}, containing a script in Pari/GP that runs this splitting process.
\end{proof}

\section{Basics of the theory of peeling} \label {sec:peel}
 
 In this section we recall, for the reader's convenience, the basics  of the  theory of peeling, referring to \cite {M,MT} for the proofs of the results we will mention and use. 
 
\subsection{A few definitions}\label{ssec:def} Let $(S,D)$ be a pair with $D$ reduced. We denote by $M(D)$ the \emph {intersection matrix} of $D$, i.e., the symmetric matrix of order $\ell=\ell(D)$ with entries $m_{ij}=D_i\cdot D_j$, for $1\leqslant i, j\leqslant \ell$. To say that $M(D)$ is  negative definite [resp. semidefinite], we write $M(D)<0$ [resp $M(D)\leqslant 0$]. The matrix $M(D)$ depends on the ordering of the components of $D$, but its rank and its being negative definite (or semidefinite) do not. 

For any curve $C\leqslant D$, one defines its \emph{branching number} in $D$ as $\beta_D(C)=C\cdot (D-C)$. If $C$ is a component of $D$, this is the valence of the vertex corresponding to $C$ in $\mathfrak G(C)$. 

A tree $D$ is a  \emph{chain} if each vertex of $\mathfrak G(D)$ has valency 1 or 2 (equivalently, $D$ is a tree with only two terminal components).  Given $(S,D)$, a chain $C\leqslant D$ which is  a connected component of $D$ is called a  \emph{rod} of $D$. 
A chain $T\leqslant D$ is called a \emph{twig} of $D$ if $\beta_D(T)=1$ and $T$ meets
$D-T$ in a point of a terminal component of $T$. 
A twig of $D$  is \emph{maximal} if it cannot be extended to a twig of $D$ with more components.  

One says that $C\leqslant D$  is a \emph{fork} if $C$ is a connected component of $D$, if all components of $C$ are smooth and rational and if $\mathfrak G(C)$ is one of the graphs in \cite [pp. 436--437] {MT}, see also \cite [Lemma 3.4.1]{M}.

A curve $D$ is said to be \emph{admissible} if all of its components $C$ are smooth, rational with $C^ 2\leqslant -2$. A smooth rational component $C$ of $D$ with $C^ 2\geqslant -1$ is called \emph{irrelevant}.

\subsection{Peeling the bark and almost minimal models}\label{ssec:amin} From now on we will consider log smooth pairs $(S,D)$ with $S$ rational and all components of $D$ rational.

Let $C\leqslant  D$ be an admissible twig, rod or fork, with $\ell(C)=\ell$ and write $C=C_1+\ldots +C_\ell$ as the sum of its irreducible components. Then  one has $M(C)<0$ and one can uniquely determine an effective $\mathbb Q$--divisor 
\[{\rm Bk}(C)=\gamma_1C_1+\ldots+\gamma_\ell C_\ell,\]
called the \emph{bark} of $C$, such that
\[
(D-{\rm Bk}(C)+K_S)\cdot C_i=0,\,\, \text{for}\,\, 1\leqslant i\leqslant \ell.
\]
One has  
\begin{equation}\label{eq:gamma}
0<\gamma_i\leqslant 1,\,\, \text{for}\,\, 1\leqslant i\leqslant \ell
\end{equation}

\begin{property}\label{proper:gamma} The 
 equality on the right in \eqref {eq:gamma} holds for some $i\in \{1,\ldots,\ell\}$ if and only if $\gamma_i=1$ for all $i\in \{1,\ldots,\ell\}$, in which case $C$ is either a rod or a fork and $C_i^ 2=-2$, for all $i\in \{1,\ldots,\ell\}$. In this case we say that $C$ is a \emph{$(-2)$--rod}  or a  \emph{$(-2)$--fork}. 
 \end{property}
 
The  process of subtracting ${\rm Bk}(C)$ from $D$ is called the \emph{peeling} of  ${\rm Bk}(C)$ out of $D$.

Consider the sets $\{T_1,\ldots,T_t\}$,  $\{R_1,\ldots,R_r\}$ and $\{F_1,\ldots F_f\}$ of maximal admissible twigs, of admissible rods and of admissible forks respectively.  These curves are all pairwise disjoint, so we can peel their barks independently out of $D$, and we obtain 
\[
D=D^ \sharp+ {\rm Bk}(D),\,\, \text {where}\,\, {\rm Bk}(D):=\sum_{i=1}^ t{\rm Bk}(T_i)+\sum_{i=1}^ r{\rm Bk}(R_i)+\sum_{i=1}^ f{\rm Bk}(F_i)\,\, \text{is called the \emph{bark} of $D$}.
\] 
\begin{property}\label{property:bark} One has:\\
\begin{inparaenum}
\item [(i)] $D^ \sharp$ is an effective (perhaps $0$) $\mathbb Q$--divisor and ${\rm Supp}({\rm Bk}(D))$ contains no smooth rational curve $C$ with $C^ 2\geqslant -1$;\\
\item [(ii)] $M({\rm Bk}(D))<0$;\\ 
\item [(iii)] 
\[
(K_S+D^ \sharp)\cdot C= 0, \,\,\, \text {for all components $C$ of} \,\,\, {\rm Supp}({\rm Bk}(D)),
\]
whereas
\[
(K_S+D^ \sharp)\cdot C\geqslant 0, \,\,\, \text {for all components $C$ of} \,\,\, D-{\rm Supp}({\rm Bk}(D)),
\]
except for irrelevant components of non--admissible twigs, rods and  forks;\\
\item [(iv)] 
\[
h^ 0(S, \mathcal O_S(n(D+K_S))=h^ 0(S, \mathcal O_S([n(D^ \sharp+K_S)]),\,\,\, \text{for every integer}\,\,\, n\geqslant 0,
\]
where $[\,\,\,]$ denotes the \emph {integral part}.
\end{inparaenum}

 \end{property}

From
\[
D+K_S=(D^ \sharp+K_S)+{\rm Bk}(D),
\]
and from the fact that ${\rm Bk}(D)$ is effective and that nef divisors are pseudo--effective, one has that 
\begin{equation}\label{eq:nef}
 D^ \sharp+K_S \,\,\, \text{nef}\,\,\, \Longrightarrow \,\,\, \kod(S,D)\geqslant 0.
\end{equation}

The pair $(S,D)$ is said to be \emph{almost minimal}
if, for every irreducible curve $C$ on $S$, either $(D^ \sharp + K) \cdot C \geqslant 0$  or $(D^ \sharp + K) \cdot C < 0$ and  $M(C + {\rm Bk}(D))$ is not negative definite. 

\begin{property}\label{proper:off} A curve $C$ which offends almost minimality of $(S,D)$ is a $(-1)$--curve which can be contracted without offending log smoothness. 
\end{property}

Hence one has:

\begin{theorem}\label{thm:am} Let $(S,D)$ be log smooth, with $S$ rational and all components of $D$ rational. There is  a log smooth, almost minimal pair $(\tilde S, \tilde D)$, such that:\begin{inparaenum}\\
\item [(i)] there is a birational morphism $\mu: S\to \tilde S$ such that $\tilde D=\mu_*(D)$;\\
\item [(ii)] 
\[
P_n(S,D)=P_n(\tilde S,\tilde D),\,\,\, \text{for every integer}\,\,\, n\geqslant 0,
\]
in particular
\[
\kod (S,D)=\kod(\tilde S, \tilde D).
\]
\end{inparaenum}

\end{theorem}

The pair $(\tilde S, \tilde D)$ of Theorem \ref {thm:am} is called an \emph{almost minimal model} of $(S,D)$. 

As at the end of \S \ref {ssec:small}, one can start with a marked triple $(S,D,\mathfrak K)$ with $(S,D)$ log smooth.  If $(S,D)$ is almost minimal, then 
$(S,D,\mathfrak K)$ will be also said to be \emph{almost minimal}. If  $(\tilde S,\tilde D)$ is an almost minimal model of $(S,D)$, with $\mu: S\to \tilde S$, we define $\tilde {\mathfrak K}$ to be the union of  ${\rm cl}_\mu(\mathfrak K)$ and of $\mathfrak K_\mu$ (recall Remark \ref {rem:aa} and \S \ref {ss:mark}). The almost minimal triple $(\tilde S,\tilde D,\tilde  {\mathfrak K})$ will be said to be  an \emph{almost minimal model} of  $(S,D,\mathfrak K)$.

\subsection{The bark contraction}\label{ssec:barkc}

By Property \ref {property:bark}, (ii) and by \cite [Lemma 5.2.3, Chapt. 1]{M}, there is a birational morphism $\phi: S\to \bar S$, such that:\\
\begin{inparaenum}
\item [(i)]  $\bar S$ is a normal projective surface, with rational singular points, $\phi$ induces an isomorphism between $S-{\rm Supp}({\rm Bk}(D))$ and $\bar S-{\rm Sing}(\bar S)$, and each singular point of $\bar S$ corresponds bijectively to a connected component of ${\rm Supp}({\rm Bk}(D))$, which is contracted there;\\
\item [(ii)]  there is a positive integer $N$ such that for any Weil divisor $Z$ of $\bar S$, $NZ$ is a Cartier divisor. For any Weil divisor $Z$ of $\bar S$, one defines the numerical equivalence class 
\[
\phi^ *(Z):=\frac {\phi^ *(NZ)}N\in {\rm NS}(S)\otimes _\mathbb Z\mathbb Q
\]
which is independent of $N$, and the intersection of Weil divisors $Z,Z'$ of $\bar S$  as
\[
Z\cdot Z'=\frac {\phi^ *(Z)\cdot \phi^ *(Z')}{N^ 2}\in \mathbb Z \frac 1{N^ 2}.
\]
Thus, the \emph{N\'eron--Severi} $\mathbb Q$--vector space ${\rm NS}(\bar S)$, consisting of the numerical equivalence classes of Weil $\mathbb Q$--divisors on $\bar S$, is well defined, and its rank is, as usual, denoted by $\rho(\bar S)$;\\
\item [(iii)]  one has
\begin{equation}\label{eq:1}
K_{\bar S}= \phi_*(K_S),\,\,\, D^ \sharp +K_S=\phi^ *(\bar D+K_{\bar S}), \,\,\, \text{where}\,\,\, \bar D:=\phi_*(D^ \sharp). 
\end{equation}
\end{inparaenum}

The pair $(\bar S, \bar D)$ is called the \emph{bark-contraction} of $(S,D)$. 

Mori theory can be developed on $\bar S$. Let $\overline {\rm NE}(\bar S)\subset 
{\rm NS}(\bar S)\otimes_\mathbb Q \mathbb R$ be the \emph {Mori cone}, i.e., the smallest convex cone (closed under multiplication by $\mathbb R_+$) containing the classes of all irreducible curves on $\bar S$. Let $\bar L$ be an ample Cartier divisor on $\bar S$. For any positive $\epsilon\in \mathbb R$, define
\[
\overline {\rm NE}_\epsilon(\bar D,\bar S):=\{Z\in \overline {\rm NE}(\bar S): (\bar D+K_{\bar S})\cdot Z\geqslant -\epsilon (\bar L\cdot Z)\}.
\]
One has:
\begin{theorem}[The Cone Theorem]\label{thm:cone} For any positive $\epsilon\in \mathbb R$, there are (not necessarily smooth) rational curves $\bar \ell_1,\ldots, \bar \ell_ r$ on $\bar S$, such that
\begin{equation}\label{eq:cone}
\overline {\rm NE}(\bar S)=\sum_{i=1}^ r \mathbb R_+[\bar \ell_i]+\overline {\rm NE}_\epsilon(\bar D,\bar S)
\end{equation}
and
\[
0>(\bar D+K_{\bar S})\cdot \bar \ell_i\geqslant -3,\,\,\, \text {for}\,\,\, 1\leqslant i\leqslant r.
\]
\end{theorem}

Since $\overline {\rm NE}(\bar S)$ is polyhedral in the half--space where $\bar D+K_{\bar S}+\epsilon \bar L<0$, one can define the concept of \emph{extremal ray} of 
$\overline {\rm NE}(\bar S)$ as in the smooth case. Thus,  if $r>0$ in \eqref {eq:cone}, we may assume that $\mathbb R_+[\bar \ell_i]$ are extremal rays for $1\leqslant i\leqslant r$. 

One has:

\begin {prop}\label{prop:extr} In the above setting, let $\mathbb R_+[\bar \ell]$ be an extremal ray of $\overline {\rm NE}(\bar S)$. Let $\ell$ be the proper transform of $\bar \ell$ on $S$. Then one of the following facts occurs:\\
\begin{inparaenum}
\item [(i)] $M(\ell+{\rm Bk}(D))\leqslant 0$ but not $M(\ell+{\rm Bk}(D))< 0$:  then $\bar \ell^ 2=0$ and, for $n\gg 0$, the linear system $|nN\bar \ell|$ is base point free and composed with a pencil $|f|$ whose general member is isomorphic to $\mathbb P^ 1$. If $\mathfrak f=\phi^ *(f)$, then $|\mathfrak f|$ is a base point free pencil of rational curves on $S$; \\
\item [(ii)] $\overline {\rm NE}(\bar S)\otimes \mathbb Q$ is generated by the class of $\bar \ell$, hence $\rho(\bar S)=1$ and $-(\bar D+K_{\bar S})$ is ample. Then $-(D^ \sharp +K_S)$ is nef and big and for an irreducible curve $C$ of $S$ one has $C\cdot (D^ \sharp +K_S)=0$ if and only if $C$ is an irreducible component of ${\rm Bk}(D)$. 
\end{inparaenum}

\end{prop}

In case (ii) of Proposition \ref {prop:extr}, one says that $(S,D)$ (or $(\bar S,\bar D)$) is a \emph{logarithmic del Pezzo surface of rank 1}. If $\bar D=0$ (equivalently, if $D={\rm Supp}({\rm Bk}(D)))$,  we say that $(S,D)$ has \emph{shrinkable boundary} (in \cite  {MT} is used a different terminology, which would be confusing  here).  

The classification of logarithmic del Pezzo surface of rank 1 is still an open problem in its generality. One has the following:

\begin {conj}[See \cite {MT}] \label{conj:MT} If $(S,D)$ is a logarithmic del Pezzo surface of rank 1 with 
shrinkable boundary, then $\bar S=\mathbb P^ 2/G$, where $G$ is a finite subgroup of ${\rm PGL}(2,\mathbb C)$.  
\end{conj} 

Some properties of logarithmic del Pezzo surfaces of rank 1 with shrinkable boundary are described in \cite [\S 4]{MT}.  Keel and McKernan gave in \cite {KM} a classification theorem (Theorem 23.2), which applies to all but a bounded family of rank one logarithmic del Pezzo surfaces.

\section{The peeling approach to the contraction theorem}\label{sec:peelthm}

In this section we prove the following:

\begin{theorem}\label{thm:main} Let $(S,D,\mathfrak K )$ be a marked triple, with $S$ rational and $\kod(S,D)=-\infty$. Then $(S,D,\mathfrak K )$ is contractible unless, perhaps, $-(D^ \sharp+K_S)$ is nef and big and one of the following occurs:\\
\begin{inparaenum}
\item [(1)] $(S,D)$ is a logarithmic del Pezzo surface of rank 1 with shrinkable boundary;\\ 
\item [(2)]  if $(\tilde S, \tilde D)$ is an almost minimal model of $(S,D)$, then a connected component of $\tilde D$ (and only one) is a non--admissible fork. 
\end{inparaenum} 
\end{theorem}

By Corollary \ref {lem:log},  $\kod(S,D)=-\infty$  implies that $(S,D)$ is log smooth and that, for any connected curve $C\leqslant D$, one has $p_a(C)=0$. Since $\mu: S\to \tilde S$ is a morphism, which does not blow--up $\mathfrak K$, we can pretend that $(S,D,\mathfrak K )$ is almost minimal (see the end of \S \ref {ssec:amin}). 

The proof of Theorem \ref {thm:main} will consist in a number of steps. The first one is to show that we are in position to apply Proposition \ref  {prop:extr}.

\begin{lemma}\label{lem:nnef} If $(S,D)$ is almost minimal and $\kod(S,D)=-\infty$, then $\bar D+K_{\bar S}$ is not nef.
\end{lemma}

\begin{proof} By \eqref {eq:nef},  $D^\sharp+K_S$ is not nef. Hence there is an irreducible curve $Z$ on $S$ such that $(D^\sharp+K_S)\cdot Z<0$. Set $\bar Z=\phi_*(Z)$ and note that $\bar Z$ is non--zero, because $Z$ is not in ${\rm Supp}({\rm Bk}(D))$ (see Property \ref {property:bark}--(iii)).
We claim that $(\bar D+K_{\bar S})\cdot \bar Z<0$. Indeed, by \eqref {eq:1}, we have
\[
(\bar D+K_{\bar S})\cdot \bar Z= \frac 1N (D^\sharp+K_S)\cdot \phi^* (N\bar Z).
\]
One has $\phi^* (N\bar Z)=NZ + E$, with $E$ contracted by $\phi$, hence ${\rm Supp}(E)\leqslant {\rm Supp}({\rm Bk}(D))$. By Property \ref {property:bark}--(iii), one has 
$(D^\sharp+K_S)\cdot E=0$, hence
\[
(\bar D+K_{\bar S})\cdot \bar Z= (D^\sharp+K_S)\cdot Z<0.
\]
\end{proof}

By Lemma \ref {lem:nnef} and  Proposition \ref  {prop:extr},  there are extremal rays in ${\rm NS}(\bar S)$. The proof of Theorem \ref {thm:main} consists in the discussion of the two cases corresponding to (i) and (ii) of  Proposition \ref  {prop:extr}.

\subsection {The proof of Theorem \ref {thm:main} in case (i) of Proposition \ref  {prop:extr}}\label{ssec:i}  As in (i) of Proposition \ref  {prop:extr}, let $\mathfrak f=\phi^ *(f)$. Since $|f|$ is base point free on $\bar S$, then $\mathfrak f\cdot {\rm Bk}(D)=0$. 
From $f\cdot (\bar D+K_{\bar S})<0$ we have $\mathfrak f\cdot (D+K_S)=\mathfrak f\cdot (D^ \sharp+K_S)<0$. Since $\mathfrak f\cdot K_S=-2$ and $\mathfrak f$ is nef, we find $\mathfrak f\cdot 
D\leqslant 1$.  The argument is now similar to the one at the end of the proof of Theorem \ref {thm:cool}. 

Indeed, by making
$\varphi_{|\mathfrak f|}: S\to \PP^ 1$ \emph{relatively minimal},  
 we have a birational morphism $f: S\to \FF_n$ for some $n\in \mathbb N$, which does not blow--up $\mathfrak K$. Set $f_*(D,\mathfrak K)=(f_*(D),\mathfrak K')$. 
Since $\mathfrak f\cdot D\leqslant 1$, we have that:\\
\begin{inparaenum}
\item [(a)] either $f_*(D)=0$,\\
\item [(b)] or $f_*(D)$ consists of finitely many distinct curves of $|\mathfrak f|$, \\
\item [(c)] or $f_*(D)$ consists of finitely many (may be 0) curves of $|\mathfrak f|$, plus an irreducible curve $\Gamma$ such that $\Gamma\cdot \mathfrak f=1$.
\end{inparaenum}

In case (a), the proof of Theorem \ref {thm:main} is finished. In case (b) we can contract $f_*(D)=\mathfrak f_1+\ldots+\mathfrak f_\ell$, with $\mathfrak f_i\in |\mathfrak f|$, for $1\leqslant i\leqslant \ell$ with a series of elementary transformations based at general points of $\mathfrak f_1,\ldots,\mathfrak f_\ell$ respectively. In case (c) we proceed as in the final part of the  proof of Theorem \ref {thm:cool}.

\subsection {The proof of Theorem \ref {thm:main} in case (ii) of Proposition \ref  {prop:extr}} Here we are in the rank 1 logarithmic del Pezzo case.  The  shrinkable boundary case corresponds to  (1) in Theorem \ref {thm:main}. So we are left with the non--shrinkable boundary case. 

From \cite [Lemma 1.14.5, Lemma 3.14.6 and Theorem 3.15.1] {M} and \cite [Lemma 11 and Theorem 12] {MT}, we have that $Y:=D-{\rm Supp}({\rm Bk}(D))$ is connected and consists of one or two irreducible components. Precisely:\\
\begin{inparaenum}
\item [(i)] If $Y$ is irreducible, then:\\
\item [(i-a)] either $Y$ is an irrelevant component of a rod $R$, hence $m:=Y^ 2\geqslant -1$, and $R-Y$ consists of one or two admissible twigs;\\
\item [(i-b)] or $Y$ is the central component of a non--admissible fork;\\
\item [(ii)] if $Y$ has two irreducible components $Y_1,Y_2$, then $Y_1\cdot Y_2=1$, $Y_1+Y_2$ belongs to a rod $R$ and $R-Y_1-Y_2$ consists of one or two admissible twigs, accordingly $Y^ 2_i\geqslant -1$ for $i=1,2$.
\end{inparaenum}

If (i-b) holds we are in case (2) of Theorem \ref {thm:main}.

Suppose we are in case (i-a). Write $R=Y+T_1+T_2$, where $T_1,T_2$ are two distinct admissible twigs, with $0\leqslant \ell(T_1)\leqslant \ell(T_2)$.   If $\ell(T_i)\neq 0$, we denote by $Y_i$ the component of $T_i$ which intersects $Y$, and by $P_i$ the intersection point of $Y_i$ with $Y$, for $i=1,2$.  

Assume  $m\geqslant 1$ and $\ell(T_i)>0$  for $i=1,2$ (otherwise the proof is similar, but easier).  
By blowing up $m-1$ general points of $Y$, we may assume that $m=1$. Then the morphism  $\phi_{|Y|}: S\to  \mathbb P^ {2}$  is birational, contracts $D-Y-Y_1-Y_2$ to points and maps $Y, Y_1,Y_2$ to three lines
$C,C_1, C_2$ respectively, and  we are reduced to the case of $(\mathbb P^ 2, C+C_1+ C_2, \mathfrak C)$. This is contractible by Corollary \ref {cor:plane}. 


Let now $m=0$, so that $|Y|$ is a base point free pencil of rational curves. By making a series of elementary transformations at $P_1$, we can increase $Y_2^ 2$ as much as we like and make it positive. Then, to finish he proof, we can argue as in the case $m\geqslant 1$.

Finally, assume $m=-1$. We note that, since $(S,D)$ is almost minimal, $M(R)$ is not negative definite. 
Then we blow down $Y$ and consecutively all $(-1)$--curves which appear in the image of $R$. Since $M(R)$ is not negative definite, after a finite number of steps we find a rod with a smooth rational component of non--negative self--intersection. Then we can finish as above.  

Case (ii) can be treated in a similar way. Thus Theorem \ref {thm:main} is proved in this case too.

\bigskip

\end{document}